\documentclass[11pt,letterpaper]{amsart}
      \usepackage{amsmath}
      \usepackage{}
      \usepackage{amsfonts}
      \usepackage{mathrsfs}
      \usepackage{amsfonts, latexsym, rawfonts, amsmath, amssymb, amsthm}
      \usepackage[plainpages=false]{hyperref}

      \usepackage{fancyhdr}

      \usepackage{pdfsync}

      \usepackage{esint}
      \usepackage{graphics, graphicx}
      \usepackage{tikz}
      \usepackage{appendix}

      \numberwithin{equation}{section}
      \newcommand{\beq}{\begin{equation}}
      \newcommand{\eeq}{\end{equation}}
      \newcommand{\beqs}{\begin{eqnarray*}}
      \newcommand{\eeqs}{\end{eqnarray*}}
      \newcommand{\beqn}{\begin{eqnarray}}
      \newcommand{\eeqn}{\end{eqnarray}}
      \newcommand{\beqa}{\begin{array}}
      \newcommand{\eeqa}{\end{array}}

      \def\lra{\longrightarrow}

      \def\bc{\begin{center}}
      \def\ec{\end{center}}

      \def\begeq{\begin{equation}}
      \def\endeq{\end{equation}}
      \def\and{\quad{\rm and}\quad}

      \let\lra=\longrightarrow

      \def\mapright\#1{\, \smash{\mathop{\lra}\limits^{\#1}}\, }

      \newtheorem{prop}{Proposition}[section]
      \newtheorem{theo}[prop]{Theorem}
      \newtheorem{lem}[prop]{Lemma}
      
      \newtheorem{cor}[prop]{Corollary}
      \newtheorem{rem}[prop]{Remark}

\begin{document}

       \title{ No compact split limit Ricci flow of  type II from the blow-down}

     \author{Ziyi  $\text{Zhao}^{\dag}$ and Xiaohua $\text{Zhu}^{\ddag}$}

 \address{BICMR and SMS, Peking
 University, Beijing 100871, China.}
 \email{ 1901110027@pku.edu.cn\\\ xhzhu@math.pku.edu.cn}

 \thanks {$\ddag$ partially supported  by National Key R\&D Program of China  2023YFA1009900 and 2020YFA0712800,  and  NSFC 12271009.}


      \begin{abstract}  By Perelman's $\mathcal L$-geodesic theory,  we  study the   blow-down solutions  on a  noncompact $\kappa$-noncollapsed   steady  gradient   Ricci soliton  $(M^n, g)$  $(n\ge 4)$ with  nonnegative curvature operator and positive Ricci curvature  away from a compact set of $M$.   We prove that any compact split ancient     solution   of codimension one   from the  blow-down of $(M, g)$ is of  type I.  The result is a generalization of  our previous  work   from  $n=4$ to any dimension.
        \end{abstract}

       \date{}

    \maketitle

   % \tableofcontents

    \setcounter{section}{-1}

    \section{Introduction}
    
Let $(M^n, g, f )$  $(n\ge 4)$ be a complete  noncompact $\kappa$-noncollapsed  steady  gradient Ricci soliton with curvature operator
      %${\rm Ric}>0$ and
        $\rm{Rm}\geq 0$ away from a compact  set $K$ of $M$.   Let $g(\cdot, t)=\phi^*_{t}(g)$ $(t\in (-\infty, \infty))$  be an induced ancient Ricci flow of  $(M,g)$,
        where $\phi_{t}$ is a family of transformations generated by the gradient vector field $-\nabla f$.
          For any  sequence of $p_i\in M$ $(\to \infty)$,  we consider the (normally) rescaled Ricci  flows $(M, g_{p_i}(t); p_i)$,  where
         \begin{align}\label{rescaling-flow}g_{p_i}(t)=r_i^{-1} g( \cdot, r_it),
    \end{align}
    $r_iR(p_i)=1$.
    By a version of  Perelman's compactness theorem for  ancient  $\kappa$-solutions \cite[Proposition 1.3]{ZZ-4d} (also see Proposition \ref{Dimension-reduction}),
    we know that $(M, g_{p_i}(t); p_i)$   converge  subsequently to a  splitting flow $(N \times \mathbb{R}, \bar {g}(t); p_\infty)$ in the Cheeger-Gromov sense, where
    \begin{align}\label{splitt-solution} \bar g(t)= h(t) +ds^2, ~ {\rm on}~ N\times \mathbb R,
     \end{align}
    and $h(t)$  ($t\in (-\infty, 0]$)  is an ancient   $\kappa$-solution on an  $(n-1)$-dimensional $N$.
    For simplicity, we call $(N,  h(t))$ a  split limit  flow of  codimension one  arising from  the  blow-down of $(M, g)$.

  Recall that
         an  ancient $\kappa$-solution $(N, h(t))$ $(t\in (-\infty, 0])$ is  of type I  if  it satisfies
                \begin{align*}
                    \sup_{N\times (-\infty,0]} (-t)|R(x,t)|<\infty.
                \end{align*}
              Otherwise,   it is called type \uppercase\expandafter{\romannumeral2},  if  it satisfies
                \begin{align*}
                    \sup_{M\times (-\infty,0]} (-t)|R(x,t)|=\infty.
                \end{align*}

In this paper,  we prove

    \begin{theo}\label{main-notype2}
        Let $(M^n, g)$ $(n\ge 4)$  be a  noncompact $\kappa$-noncollapsed steady gradient Ricci soliton which
        satisfies
         \begin{align}\label{condtion}{\rm Rm} \geq 0~ {\rm and }~{\rm Ric}>0~{\rm  on}~ M\setminus K,
         \end{align}
         where  ${\rm Rm}$ and ${\rm Ric}$ denote the curvature operator and Ricci curvature of $(M, g)$, respectively.
          Then any  compact split ancient  solution   $(N, h(t))$ in  (\ref{splitt-solution}) from the  blow-down of $(M, g)$ is of  type I. In the other words, there is no    $(n-1)$-dimensional   compact split ancient  solution of
        type II  from the  blow-down  of $(M, g)$.
    \end{theo}

    Since  any  compact  ancient $\kappa$-solution   of type I  is a  gradient  shrinking  Ricci soliton, which has been classified in  \cite[Theorem 7.34]{CLN} and  \cite{BW, Ni} (also see \cite[Proposition 4.1]{ZZ-4d}),  Theorem \ref{main-notype2} actually  gives a classification of  all    compact  split  blow-down  solutions of codimension one  of $(M, g)$ ,  which satisfies (\ref{condtion}).

    As an application of Theorem \ref{main-notype2}, we prove the  following alternative principle.

   \begin{cor}\label{alter-coro}
       Let $(M^n,  g)$ be a  steady  gradient   Ricci soliton in Theorem \ref{main-notype2}.    Then either all   split  blow-down  solutions  $(N, h(t))$  of $(M, g)$ in  (\ref{splitt-solution}) are  $(n-1)$-dimensional  compact ancient $\kappa$-solution of type I, or  $(n-1)$-dimensional   noncompact ancient $\kappa$-solution.
   \end{cor}

 Corollary \ref{alter-coro} confirms a conjecture    \cite[Conjecture 4.5]{ZZ-4d}  and is a     generalization of    \cite[Theorem  0.2]{ZZ-4d}   hereby  from $n=4$ to any dimension. We would like to mention that
          our previous proof  in the case $n=4$   heavily depends    on a  deep  classification result   for $3d$ compact $\kappa$-solutions of type II  proved by Brendle-Daskalopoulos-Sesum \cite{BDS}.

      In order to follow the argument in \cite{ZZ-4d}  to prove   Theorem \ref{main-notype2},  we  use Perelman's $\mathcal L$-geodesic theory   in \cite[Section 7, Section 9]{P}  to construct a  limit gradient   shrinking  Ricci soliton through a  sequence of  normally rescaled Ricci flows (cf. Section  4, 5).  Inspired by  a recent work of  Bamler-Chan-Ma-Zhang \cite{BCMZ},
       we are able  to choose a  suitable base point in each level set of   the potential function $f$  such that the corresponding  $\ell$-length is  uniformly bounded (cf. Section 2).   Unfortunately,   we  can't use the Perelman's   result directly to get the gradient estimate for the (rescaled)  reduced distance  function $\ell_i(x,\tau)$  since   the nonnegative  curvature condition just holds  outside   a compact set of $M$  (see Remark \ref {remark}).   We will  do the curvature decay estimates  to overcome the difficulty (see Section 3).

       To see some examples of  steady gradient Ricci solitons $(M, g)$ which satisfy  the condition  (\ref{condtion}) in  Theorem \ref{main-notype2},  we refer the reader to \cite{Bry, App, Yi-flyingwings}, etc.

    According to   the proof of Theorem \ref{main-notype2},  we can get the following explicit  curvature decay estimate.

   \begin{theo}\label{compact-linear-decay}
    Let $(M,  g)$    be a  steady  gradient   Ricci soliton in Theorem \ref{main-notype2}.   Suppose that there exists a sequence of rescaled Ricci flows $(M,g_{p_i}(t); p_i)$, which converges subsequently to a  splitting Ricci flow $(N\times \mathbb R, \bar g(t); p_\infty)$ as in (\ref{splitt-solution}) for some compact ancient $\kappa$-solution $(N, h(t))$.  Then the scalar curvature of $(M,  g)$ decays to zero linearly. Namely,  there exist  two positive constants $C_1$ and $C_2$ such that
       \begin{align}\label{R-decay-linearly}
           \frac{C_1}{\rho(x)}\leq R(x)\leq \frac{C_2}{\rho(x)}, ~\rho(x)\gg 1,
       \end{align}
    where  $R(x)$ is the scalar  curvature of  $(M,  g)$ and  $\rho(x)$  denotes   a  distance function from a fixed point  on  M.
   \end{theo}

Theorem \ref{compact-linear-decay} is  an improvement of  \cite[Lemma 2.2]{ZZ-4d} (also see (\ref{R-decay})).     We notice   that complete  noncompact $\kappa$-noncollapsed  steady  gradient   Ricci solitons   with nonnegative  curvature under the condition  (\ref{R-decay-linearly})  have  been  classified by  Deng-Zhu \cite{DZ3, DZ-JEMS, DZ-SCM}.

 The paper is organized as follows. In Section 1, we  first review a  compactness  theorem  for  normally rescaled Ricci flows of $(M, g)$  in \cite{ZZ-4d} (cf. Proposition \ref{Dimension-reduction}),  then we recall the Perelman's $\mathcal L$-geodesic theory and translate it for a steady  gradient Ricci soliton.
   In Section 2, we  use  a method  in \cite{BCMZ} to  prove  the existence of $\ell$-centers in each level set of the potential function $f$.  In Section 3,   we do the curvature estimates for $\ell$-centers (cf.  Lemma  \ref{upper-boundofR} and Proposition \ref{ell-curvature-lower}).   In Section 4,  we get the gradient estimate for the  (rescaled)  reduced distance  function $\ell_i(x,\tau)$ and then construct a limit  shrinking Ricci solitons through the normally rescaled Ricci flows of $(M, g)$ (cf. Proposition \ref{local-ell-estimate} and Proposition \ref{Pel-shrinker}).  Both of  Theorem \ref{main-notype2} and Theorem   \ref{compact-linear-decay}  will be   proved in Section 5.

Since $3d$ complete  noncompact $\kappa$-noncollapsed  steady  gradient   Ricci solitons have  been  classified by Brendle motivated by  a conjecture of Perelman \cite{Bre-3d}, we will always assume  that  the solitons  are of  dimension $n\ge 4$  and  not Ricci flat  in this paper.  

\vskip5mm

$\mathbf {Acknowledgments.}$  The authors would like to thank  the referee for many valuable comments   on   their  paper.

\section{Preliminaries}

A complete Riemannian metric $g$ on  $M$ is  called a gradient Ricci soliton if there exists a smooth function $f$ (which is called a potential  function)  on $M$ such that
  \begin{equation}\label{Def-soliton}
  R_{ij}(g)+\rho g_{ij}=\nabla_{i}\nabla_{j}f,
  \end{equation}
  where $\rho$ is a constant. The gradient Ricci soliton is called expanding,  steady and shrinking  according to  $\rho >,  = ,  <0$,  respectively.     These three types  of   Ricci solitons  correspond to three different blow-up solutions  of Ricci flow  \cite{Ham-singular}.

   In  the case of  steady Ricci solitons, we can rewrite (\ref{Def-soliton}) as
    \begin{align}\label{soliton-equation} 2\operatorname{Ric}(g)=L_Xg,
    \end{align}
    where  $L_X $ is the Lie operator along   the gradient vector field  (VF)  $X = \nabla f$  generalized by $f$.
  Let $\{\phi^*_t\}_{t\in(-\infty,\infty)}$ be   the  one-parameter subgroup of transformations   generated by $-X$.   Then
  $g(t)=\phi^*_t(g)$ ($t\in(-\infty,\infty))$ is a solution of Ricci flow. Namely,  $g(t)$ satisfies
  \begin{align}\label{ricci-equ}\frac{\partial g}{\partial t}=-2{\rm Ric}(g), ~ g(0)=g.
  \end{align}
  For simplicity,  we call  $g(t)$  the soliton  Ricci flow of  $(M, g)$.

   By (\ref{soliton-equation}), we have
   \begin{align}\label{R-monotonicity}
       \langle\nabla R,\nabla f\rangle=- 2\operatorname{Ric}(\nabla f,\nabla f),
   \end{align}
   where $R$ is the scalar curvature of $g$.
   It follows
   $$R+|\nabla f|^2={\rm Const.}$$
    Since  any (non Ricci-flat) steady gradient Ricci soliton  has  positive  scalar curvature  (\cite{Zh, Ch}),  the above equation can be normalized by
   \begin{align}\label{scalar-equ} R+|\nabla f|^2=1.
   \end{align}

In this paper we always assume that $(M^n, g)$  $(n\ge 4)$ is a noncompact $\kappa$-noncollapsed  steady  gradient Ricci soliton with nonnegative curvature operator
     %${\rm Ric}>0$ and
       $\rm{Rm}\geq 0$ away from a compact set  $K$ of $M$.

        The following splitting theorem was proved  in  \cite[Proposition 1.2]{ZZ-4d},  which can be regarded as a version of  Perelman's compactness theorem for  higher dimensional  ancient  $\kappa$-solutions \cite{P, KL}.

 \begin{prop}\label{Dimension-reduction}
     Let $(M^n, g)$  $(n\ge 4)$ be a noncompact $\kappa$-noncollapsed  steady  gradient Ricci soliton with
     %${\rm Ric}>0$ and
       $\rm{Rm}\geq 0$ on  $M\setminus K$.   Let   $p_i\rightarrow \infty$ and $(M, g_{p_i}(t); p_i)$  a sequence of rescaled  flows with   $ R_{p_i}\left(p_i, 0\right)=1$ as in (\ref{rescaling-flow}) .  Then $(M, g_{p_i}(t); ~$
       $p_i)$ subsequently converge to a  splitting flow $(N \times \mathbb{R}, \bar {g}(t)=h(t) +ds^2; p_\infty)$ as
      in  (\ref{splitt-solution}), where  $(N, {g}(t))$ is an   ancient $\kappa$-solution of codimension one.    Moreover,   for $n=4$, $\rm{Rm}\geq 0$ can be weakened to the sectional curvature  $\rm{Km}\geq 0$  on  $M\setminus K$.

 \end{prop}

 By Proposition \ref{Dimension-reduction}, we have the following Laplace estimate for the scalar curvature of $(M, g)$, which will be used in Section 3.

  \begin{lem}\label{uniform-shi}
     Let $(M,g)$ be the $\kappa$-noncollapsed steady gradient Ricci soliton with ${\rm Rm} \geq 0$ on $M\setminus K$. Then there exists a uniform constant $C>0$, such that
     \begin{align}\label{laplace-est}
         \frac{|\Delta R(p,t)|}{R^2(p,t)}\leq C,~\forall (p,t)\in M\times (-\infty, 0].
     \end{align}
 \end{lem}

 \begin{proof}
     On the contrary, if (\ref{laplace-est}) is not true,
there  is a sequence of $p_i$ and $t_i$ such that
     \begin{align}\label{Delta-R-infty}
         \frac{|\Delta R(p_i,t_i)|}{R^2(p_i,t_i)}\to \infty.
     \end{align}
     Then we consider   rescaled flows $(M, g_i(t); p_i)$, where
     \begin{align*}
         g_i(t) = R(p_i,t_i)g(R^{-1}(p_i,t_i)t+t_i).
     \end{align*}
     By the isometry $(M,g(t_i);p_i )\cong (M, g, \phi_{t_i}(p_i))$, it is easy to see
       \begin{align}
         (M,g_i(t);p_i ) \cong (M, g_{\phi_{t_i}(p_i)}(t) ; \phi_{t_i}(p_i)),
     \end{align}
     where
     \begin{align*}
         g_{\phi_{t_i}(p_i)}(t) = R_g(\phi_{t_i}(p_i))g(R_g^{-1}(\phi_{t_i}(p_i))t).
     \end{align*}
     We may assume that  $\phi_{t_i}(p_i)\to\infty $, otherwise,
           \begin{align*}
     \frac{|\Delta R(p_i,t_i)|}{R^2(p_i,t_i)} = \frac{|\Delta_g R_g(\phi_{t_i}(p_i))|_g}{R_g^2(\phi_{t_i}(p_i))}\leq C_0,
     \end{align*}
     which contradicts to \eqref{Delta-R-infty}.

     By Proposition \ref{Dimension-reduction},  the scaled Ricci flows $(M, g_{\phi_{t_i}(p_i)}(t) ; \phi_{t_i}(p_i))$ converge  in the Cheeger-Gromov sense. Thus
     \begin{align*}
         \left|\operatorname{Rm}\left(q, g_{\phi_{t_i}(p_i)}(t)\right)\right| \leq C_1, \forall(q, t) \in B\left(\phi_{t_i}(p_i), 1 ; g_{\phi_{t_i}(p_i)}(-1)\right) \times[-1,0] .
     \end{align*}
 By the Shi's estimate \cite{Shi}, we get
 \begin{align*}
     \left|\Delta R\left(q, g_{\phi_{t_i}(p_i)}(t)\right)\right| \leq C_2, \forall (q,t) \in B\left(\phi_{t_i}(p_i), \frac{1}{2} ; g_{\phi_{t_i}(p_i)}(-1)\right) \times[-\frac{1}{2}, 0].
 \end{align*}
 In particular,
 \begin{align*}
     \left|\Delta R\left(\phi_{t_i}(p_i), g_{\phi_{t_i}(p_i)}(0)\right)\right|=\frac{|\Delta_g R_g(\phi_{t_i}(p_i))|_g}{R_g^2(\phi_{t_i}(p_i))}=\frac{|\Delta R(p_i,t_i)|}{R^2(p_i,t_i)}\leq C_2,
 \end{align*}
 which is  a contradiction with \eqref{Delta-R-infty}. The lemma is proved.
 \end{proof}

We  note  that  an ancient $\kappa$-solution is a  $\kappa$-noncollapsed  solution of Ricci flow (\ref{ricci-equ})   with  ${\rm R_m}(\cdot, t)\ge 0$ defined for any $t\in (-\infty,  0]$. The purpose of paper is to use Perelman's $\mathcal L$-geodesic theory  to study the geometry of   split  ancient $\kappa$-solution $(N,  {g}(t))$ in Proposition \ref{Dimension-reduction}.

\subsection{$\mathcal L$-length and $\mathcal L$-geodesic}

 Let $(M, \hat g(t))$ $(t\in [0, \infty)$)   be a backward Ricci flow on $M$,  namely, $\hat g(t)$ satisfies
  \begin{align}\label{backricci-equ}\frac{\partial \hat g}{\partial t}=2{\rm Ric}(\hat g), ~ \hat g(0)=g.
  \end{align}
For any $\tau > 0 $ and any piecewisely smooth curve $\Gamma : [0, \tau] \to M$ with $\Gamma(0) = o$ and $\Gamma(\tau) = p$,   $\mathcal{L}$-length of  $\Gamma$ is  defined by
 (cf.  Perelman \cite[Section 7]{P}),
 \begin{align}\label{L-length-soliton-0-hat}
  \mathcal{L}(\Gamma):=\int_0^\tau \sqrt{s}\left(R_{\hat g_{s}}+|\dot{\Gamma}|_{\hat g_{s}}^2\right)(\Gamma(s)) ds.
\end{align}
Then for any pair  $(x, \tau)$,
we define a function (called  $\mathcal L$-distance function) by
$$
L\left(x, \tau\right):=\inf _{\Gamma} \mathcal{L}(\Gamma),
$$
 The above infimum is taken for any   piecewise  smooth curve $\Gamma : [0, \tau] \to M$ with $\Gamma(0) = o$ and $\Gamma(\tau) = x$.  By Perelman \cite[Section 7]{P},   the  infimum  can be   attained  by a smooth curve which is  called a minimal  $ \mathcal L$-geodesic for the pair $(x, \tau)$.

Set
 \begin{align}\label{l-length}
\ell(x, \tau):=\frac{1}{2 \sqrt{\tau}} L(x, \tau),
\end{align}
which is called the  reduced distance of  backward Ricci flow $(M, \hat g(t))$.
 By \cite[Section 7.1]{P},  we know that  for any $\tau>0$, there exists a point $x\in M$ such that $\ell\left(x, \tau\right) \leq n / 2$. Any such point $x$ is called an $\ell$-center of $(M, \hat g(t))$  at time $\tau$.

The following lemmas was proved by  Perelman's work \cite{P}.

 \begin{lem}(\cite{P}, also see \cite[Section 18]{KL})\label{first-order-variation}
    Let $\Gamma(s) $  $(s\in [0,\tau])$  be a minimal $\mathcal {L}$-geodesic with $\Gamma(0) = o$ and $\Gamma(\tau) = x$ on $(M, \hat g(t))$ $(t\in [0, \infty)$). Let $Y(\tau) = \frac{d\Gamma}{ds}(\tau)$.  Then we have
    \begin{align}\label{grad-ell}
        |\nabla L|^2(x,\tau)& = 4\tau |Y(\tau)|^2(x)\notag\\
        & = -4\tau R(x, -\tau) + 4\tau (R(x, -\tau) + |Y(\tau)|^2(x))
    \end{align}
    and
    \begin{align}\label{ell-harnack}
        \tau^{\frac{3}{2}}\left(R(x, -\tau)+|Y(\tau)|^2(x)\right)
        =-K(x,\tau) + \frac{1}{2}L(x,\tau),
    \end{align}
    where
    \begin{align}\label{K-formula}
        K(x,\tau) = \int_0^\tau s^{\frac{3}{2}}(-\frac{\partial R}{\partial \tau}-2\langle Y, \nabla R\rangle+2 \operatorname{Ric}(Y, Y) - \frac{1}{s}R)(\Gamma(s)) ds.
    \end{align}
\end{lem}

\begin{lem}(\cite{P}, also see \cite[Lemma 2.19]{Ye})\label{ell-differential-equation}
    \begin{align}\label{ell-equ}
        \frac{\partial \ell}{\partial \tau}-\frac{R}{2}+\frac{|\nabla \ell|^2}{2}+\frac{\ell}{2 \tau}=0,
    \end{align}
    \begin{align}\label{ell-geq}
        \frac{\partial \ell}{\partial \tau}-\Delta \ell+|\nabla \ell|^2-R+\frac{n}{2 \tau} \geq 0,
    \end{align}
    and
    \begin{align}\label{ell-leq}
        \Delta \ell-\frac{|\nabla \ell|^2}{2}+\frac{R}{2}+\frac{\ell-n}{2 \tau} \leq 0 .
    \end{align}
    Moreover, \eqref{ell-geq} becomes an equality at a point if and only if \eqref{ell-leq} becomes an equality at that point.
\end{lem}

\subsection{$\mathcal L$-length associated to  a steady Ricci soliton}

 Under the  variable change $t=-\tau$,  $\hat g(\tau)=g(-t)$  $(t\in (-\infty, 0])$ becomes a  backward Ricci flow on $[0, \infty)\times M$.
 Then    the length  $ \mathcal{L}(\Gamma)$ and  $\ell(x, \tau)$  associated to  $(M, \hat g(\tau))$  can be  translated  into ones for the soliton Ricci flow $(M, g(t))$, respectively.   Actually,  we have
 \begin{align}
 \label{L-length-soliton-0}
  \mathcal{L}(\Gamma)=\int_0^\tau \sqrt{s}\left(R_{g{(-s)}}+|\dot{\Gamma}|_{g{(-s)}}^2\right)(\Gamma(s)) ds.
\end{align}
Since the scalar curvature $R(\cdot)$ is always positive  by a result of Chen \cite{Ch}, the integral function in  (\ref{L-length-soliton-0}) is positive.

By the  isometry  between  $(\Gamma(s),    g(-s))$ and  $( \phi_{-s}(\Gamma(s)) ,  g = g(0) )$ for all $s\in[0,\tau]$,  (\ref{L-length-soliton-0}) becomes
\begin{align}
    \mathcal{L}(\Gamma)=\int_0^\tau \sqrt{s}\left(R_{g}+|(\phi_{-s})_*(\dot{\Gamma}(s))|_{g}^2\right)(\phi_{-s}(\Gamma(s))) ds.\notag
\end{align}
Let $\gamma(s) = \phi_{-s}(\Gamma(s))$, then
\begin{align*}
    \dot{ \gamma}(s) = \nabla f|_{\gamma} + (\phi_{-s})_*(\dot{\Gamma}(s)).
\end{align*}
It follows
\begin{align}\label{L-length-stable}
    \mathcal{L}(\Gamma)=\int_0^\tau \sqrt{s}\left(R_{g}+|\dot{\gamma}(s)-\nabla f|_{g}^2\right)(\gamma(s)) d s.
\end{align}
Hence, the integral function in  (\ref{L-length-stable}) is just for $\gamma(s)$-curve in $M$ with the fixed soliton metric $g$.
 Without confusion, we also call $\gamma(s)$ ($s\in [0,\tau]$) a minimal  $\mathcal L$-geodesic with $\gamma(0)=o$ and $p=\gamma(\tau)=\phi_{-\tau}(\Gamma(\tau))$ as long  as $(\Gamma(s), g(-s))$ is   a minimal  $\mathcal L$-geodesic with $\Gamma(0)=o$ and $x=\Gamma(\tau)$.

 As in  \cite{BCMZ},
 we write  the reduced distance $\ell(x, \tau)$ as
\begin{align}\label{lamda-function}\lambda(p,\tau) = \ell(\phi_\tau(p),\tau).
\end{align}
Thus
\[\lambda(p,\tau) = \lambda(\phi_{-\tau}(x), \tau) = \ell(x,\tau) \leq \frac{n}{2},\]
if $x$ is an $\ell$-center at time $-\tau$.
In the following we always use $\lambda(p,\tau)$  (or  $\ell(x, \tau)$)  to  study the location of  $p$  instead of  $\ell$-center  $x$.  Without confusion, we also call  $p$ a  $\ell$-center  as long as
\begin{align}\label{uniform-l}
\lambda(p,\tau)  \leq  A_0,
 \end{align}
 where $A_0\ge \frac{n}{2}$ is a fixed  constant, which will be  determined in  next section.

By a parameter change $u = \sqrt{s}$,  \eqref{L-length-stable} can be also written as
\begin{align}\label{L-length-u}
    \mathcal{L}(\Gamma):=\int_0^{\sqrt{\tau}} \left(2u^2R+\frac{1}{2}|\dot{\bar\gamma}-2u\nabla f|^2\right) d u,
\end{align}
where  $\bar\gamma(u) = \gamma(u^2)$. (\ref{L-length-u}) will be used often in next sections.

\section{Location of $\ell$-center}

In this section, we study the location of $\ell$-center $p_\tau$  which satisfies (\ref{uniform-l}) by the method in \cite{BCMZ}. From this section, we always assume that the potential function $f$ of  steady Ricci  soliton  $(M, g)$ satisfies
 \begin{align}\label{linear-f}
       c_1\rho(x)\le f(x)\le c_2\rho(x),
   \end{align}
   where  $c_1$ and $c_2$ are  two constants.
   Then we may further assume $f(o)=0$ and $f(p)\ge 0$ for any $p\in M$.  For each $\tau>0$, we   define a level set  of $f$ by
   $$\Sigma_\tau=\{p\in M|~f(p)=\tau\}.$$
   Clearly,  $\Sigma_\tau$ is compact.
    In the last section,  we will verify the condition (\ref{linear-f}) to prove the main results in  Introduction 0.

We first observe

\begin{lem}\label{lambda-o}
    $ \lambda(o,\tau)\geq \frac{\tau}{3}$ , for any $\tau>0$.
\end{lem}

\begin{proof}We note  that all $\Gamma$-curves are  loops at $o$  associated to the length $ \lambda(o,\tau)$.
   Then by (\ref{L-length-u}) together with  \eqref{scalar-equ}, we get
\begin{align}\label{L-length-equ}
\int_0^\tau \sqrt{s}  \left(R+|\dot{\gamma}-\nabla f|^2\right)& =\int_0^{\sqrt{\tau}}\left(\frac{1}{2}|\dot{\tilde{\gamma}}-2 u \nabla f|^2+2 u^2 R(\tilde{\gamma}(u))\right) d u \notag\\
& =\int_0^{\sqrt{\tau}}\left(\frac{1}{2}|\dot{\tilde{\gamma}}|^2-2 u(f \circ \tilde{\gamma}-f(o))^{\prime}+2 u^2\right) d u \notag\\
& =\frac{2}{3} \tau^{3 / 2}+\int_0^{\sqrt{\tau}}\left(\frac{1}{2}|\dot{\tilde{\gamma}}|^2+2f \circ \tilde{\gamma}(u)\right) d u\notag\\
&\geq \frac{2}{3} \tau^{3 / 2}.
\end{align}
Thus the lemma comes from (\ref{l-length}) immediately.
\end{proof}

The following lemma is due to \cite[Lemma 2.2 and Lemma 2.3]{BCMZ}.

\begin{lem}\label{lambda-distance}
    There is a universal constant $\alpha \in (0,1)$ such that for any $\tau\gg1$ and any $p'$ with $\lambda(p',\tau)\leq \frac{n}{2}$, it holds
  \begin{align}\label{distance-estimate}  {\rm dist}_g(o,p')\geq \alpha \tau.
  \end{align}
\end{lem}

\begin{proof}
      Let $\gamma_1:[0, \tau] \rightarrow M$ be a minimizing  $\mathcal L$-geodesic from $o$ to $p'$, and  $\tilde{\gamma}_2:[\sqrt{\tau},(1+\delta) \sqrt{\tau}] \rightarrow M$ be a minimizing $g$-geodesic from $p'$ to $o$ such that $|\frac{d}{ds}\tilde{\gamma}_2|_g=\delta$.
     %  constant speed for some constant $\delta>0$.
     Set $\gamma_2:\left[\tau,(1+\delta)^2 \tau\right] \rightarrow M$ by $\gamma_2(s)=\tilde{\gamma}_2(\sqrt{s})$.
     Then $\gamma_1\cup\gamma_2$ is a   special loop at $o$ with parameter $s\in [0, (1+\delta)^2 \tau]$.
     By (\ref{L-length-stable}) and (\ref{L-length-u}),  we compute
\begin{align*}
&L\left(o,(1+\delta)^2 \tau\right) \\
& \leq \int_0^\tau \sqrt{s}\left(R+\left|\dot{\gamma}_1-\nabla f\right|^2\right) d s+\int_\tau^{(1+\delta)^2 \tau} \sqrt{s}\left(R+\left|\dot{\gamma}_2-\nabla f\right|^2\right) d s \\
&= L(p', \tau)+\int_{\sqrt{\tau}}^{(1+\delta) \sqrt{\tau}}\left(\frac{1}{2}\left|\dot{\tilde{\gamma}}_2\right|^2+2 u\left|\dot{\tilde{\gamma}}_2\right||\nabla f|+2 u^2\right) d u \\
& \leq L(p', \tau)+\int_{\sqrt{\tau}}^{(1+\delta) \sqrt{\tau}}\left(\left|\dot{\tilde{\gamma}}_2\right|^2+4 u^2\right) d u \\
& \leq L(p', \tau)+\frac{\operatorname{dist}^2(o, p')}{\delta \sqrt{\tau}}+4 \frac{(1+\delta)^3-1}{3} \tau^{3 / 2} \\
& \leq L(p', \tau)+\frac{\operatorname{dist}^2(o, p')}{\delta \sqrt{\tau}}+10 \delta \tau^{3 / 2} .
\end{align*}
Note that $\lambda(p',\tau)\leq \frac{n}{2}$.  Thus  for $\tau\gg1$, we get
\begin{align}\label{lambda-triangle}
    \lambda(o,(1+\delta)^2 \tau)&\leq \frac{1}{1+\delta}\lambda(p',\tau)+\frac{\operatorname{dist}^2(o, p')}{2\delta(1+\delta) \tau}+5\frac{\delta}{1+\delta}\tau\notag\\
    &\leq \frac{\operatorname{dist}^2(o, p')}{2\delta \tau} + 10\delta \tau.
\end{align}
By  Lemma \ref{lambda-o},  it follows
\begin{align*}
    \frac{(1+\delta)^2\tau}{3}\leq\lambda(o,(1+\delta)^2 \tau)\leq \frac{\operatorname{dist}^2(o, p')}{2\delta \tau} + 10\delta \tau.
\end{align*}
Hence by taking $\delta\ll 1$,  we obtain (\ref{distance-estimate}) for some small $\alpha$.
\end{proof}

By Lemma \ref{lambda-distance}, we give  the following    location  estimate of $\ell$-center   $p_\tau$ in the level set $\Sigma_\tau$.

  \begin{prop}\label{levelset-ell}
    For any $\tau\gg 1$, there is a $p_\tau\in\Sigma_\tau$, and $\tau_0\in[c\tau,C\tau]$ such that $\lambda(p_\tau,\tau_0)\leq A_0$, where $c,C$ and $A_0$ are uniform constants.
\end{prop}

\begin{proof} Let  $c_1$ be the constant  in \eqref{linear-f}.
    Let $\gamma: [0, \frac{\tau}{ c_1\alpha}] \rightarrow M$ be a minimizing $\mathcal L$-geodesic from $o$ to $p=\gamma({\frac{\tau}{ c_1\alpha}})$ such that
    $$\lambda(p, \frac{\tau}{ c_1\alpha})\le  \frac{n}{2}.$$
        Then by Lemma \ref{lambda-distance},  we see that
     $$f(p)\geq c_1\operatorname{dist}_g(p, o) \geq \tau.$$
       Set
$$
\tau_0:=\sup \{s \in[0, \frac{\tau}{ c_1\alpha}]: f(\gamma(s)) \leq \tau\}, \quad p_\tau:=\gamma\left(\tau_0\right) .
$$
Thus  $\tau_0\leq \frac{\tau}{c_1\alpha}=C\tau$. We need to estimate the lower bound of $\tau_0$.

  Define $\tilde{\gamma}:[0, \sqrt{\frac{\tau}{ c_1\alpha}}] \rightarrow M$ by $\tilde{\gamma}(u)=\gamma\left(u^2\right)$. Then as  is \eqref{L-length-equ}, we have
   \begin{align*}
      L(p, \frac{\tau}{ c_1\alpha})&\geq \int_0^{\sqrt{\tau_0}} \left (\frac{1}{2}|\dot{\tilde{\gamma}}|^2-2 u(f \circ \tilde{\gamma}-f(o))^{\prime}  \right) du\\
      &= \int_0^{\sqrt{\tau_0}}\left(\frac{1}{2}|\dot{\tilde{\gamma}}|^2-2u\frac{d}{du}f(\tilde{\gamma}(u)) \right)du\\
      &= \int_0^{\sqrt{\tau_0}}\left(\frac{1}{2}|\dot{\tilde{\gamma}}|^2+2f(\tilde{\gamma}(u))\right)du -2\sqrt{\tau_0}\tau\\
      &\geq \int_0^{\sqrt{\tau_0}}\frac{1}{2}|\dot{\tilde{\gamma}}|^2 du -2\sqrt{\tau_0}\tau.
  \end{align*}
   It follows
   \begin{align*}
       \frac{1}{2} \int_0^{\sqrt{\tau_0}}|\dot{\tilde{\gamma}}|^2  \leq L(p, \frac{\tau}{ c_1\alpha})+2\sqrt{\tau_0}\tau.
   \end{align*}
Consequently,
$$\frac{1}{2} \int_0^{\sqrt{\tau_0}}|\dot{\tilde{\gamma}}|^2  \leq n \sqrt{\frac{\tau}{ c_1\alpha}}+2\sqrt{\tau_0}\tau .$$
Thus
\begin{align*}
&\frac{c_1^2}{2} \tau^2  = \frac{c_1^2}{2} f^2(p_\tau)\\
& \leq\frac{1}{2} \operatorname{dist}\left(o, p_\tau\right)^2 \leq \frac{1}{2}\left(\int_0^{\sqrt{\tau_0}}|\dot{\tilde{\gamma}}|\right)^2 \leq \frac{1}{2} \sqrt{\tau_0} \int_0^{\sqrt{\tau_0}}|\dot{\tilde{\gamma}}|^2 \\
& \leq n \sqrt{\frac{\tau_0\tau}{ c_1\alpha}}+2\tau_0\tau \leq \frac{c_1^2}{4} \tau^2+2\tau_0\tau.
\end{align*}
This proves
 \begin{align}\label{lowe-bound}
 \tau_0 \geq  \frac{c_1^2}{8} \tau=c\tau
 \end{align}
  for some dimensional constant $c>0$.

  Note  that $\gamma: [0, \tau_0] \rightarrow M$ is also  a minimizing $\mathcal L$-geodesic from $o$ to $p_\tau$.
  Thus, by  (\ref{lowe-bound}),  we  get
$$
\lambda\left(p_\tau, \tau_0\right) \leq \frac{\sqrt{\tau / c_1\alpha}}{\sqrt{\tau_0}} \lambda(p, \frac{\tau}{ c_1\alpha}) \leq \frac{n}{2 \sqrt{cc_1 \alpha}}=A_0.
$$
\end{proof}

\section{Curvature estimates for $\ell$-center}

In this section, we give the  curvature estimates  for $\ell$-center $p_r$ determined in Proposition \ref{levelset-ell}.

Let $\gamma(s): [0, \tau_0] \to M$  be a  minimal $\mathcal {L}$-geodesic with $\gamma(0) = o$ and $\gamma(\tau_0) = p_\tau$, where $p_\tau$ is an  $\ell $-center as in Proposition \ref{levelset-ell}.  Since $p_\tau \in \Sigma_\tau$,   we see that  $\gamma(s)\cap  K \neq \emptyset$ when $\tau\gg 1$.  Set
\begin{align}\label{tau'}
    \tau' = \inf\{ t | ~\gamma(s)\in M\setminus  K, \forall s>t\}
\end{align}
and $o_\tau' = \gamma(\tau')$.  Then the restricting  $\gamma(s)$  on $[\tau', \tau_0]$, denoted by  $\gamma_2(s)$,  is a minimal $\mathcal {L}$-geodesic between $o_\tau'$ and $p_\tau$.   Thus  it satisfies that ${\rm Rm} \geq 0$, in particular  ${\rm Ric}\ge 0$ on   $\gamma_2(s)$.
As in (\ref{lamda-function}),  we denote   a $\lambda$-function  starting from  $o_\tau'$  by  $\lambda_{o_\tau'}$.

\begin{lem}\label{ell-compact-set}
    Let $\tau'$ and $\lambda_{o_\tau'}$ defined as above for  $\tau'\gg 1$.  Then
    \begin{align}\label{l-bound}
    \frac{\tau'}{\tau_0}\to 0, ~{\rm as}~  \tau_0 \to \infty.
    \end{align}
     and $\lambda_{o_\tau'}(p_\tau, \tau_0-\tau')\leq C_0 $ for some uniform constant $C_0>0$.
\end{lem}

\begin{proof}We first prove the first assertion  (\ref{l-bound}).
    Suppose that  (\ref{l-bound})    fails.  Then  there exists a constant $\epsilon>0$ such that
    \begin{align}\label{tau'-tau0}
        \tau'\geq \epsilon \tau_0.
    \end{align}
   Since $\lambda(p_\tau, \tau_0)\leq A_0$, by \eqref{L-length-stable}, we have
    \begin{align*}
        L(p_\tau, \tau_0) &= \int_0^{\tau_0} \sqrt{s}\left(R_{g}+|\dot{\gamma}(s)-\nabla f|_{g}^2\right)(\gamma(s)) d s \leq 2A_0\sqrt{\tau_0}.
    \end{align*}
    By \eqref{scalar-equ},  it follows
    \begin{align}\label{L-ell-bound}
        L(p_\tau, \tau_0) &= \int_0^{\sqrt{\tau_0}} \left(2u^2R+\frac{1}{2}|\dot{\bar\gamma}-2u\nabla f|^2\right) d u\notag\\
        &= \int_0^{\sqrt{\tau_0}} \left( 2u^2 + \frac{1}{2}|\dot{\bar \gamma}|^2 - 2u\frac{d}{du} f(\bar \gamma(u)) \right)du\notag\\
        &= \int_0^{\sqrt{\tau_0}} \left( 2u^2 + \frac{1}{2}|\dot{\bar \gamma}|^2 + 2f(\bar \gamma(u))\right)du - 2\sqrt{\tau_0}f(p_\tau)\notag\\
        &\leq 2A_0\sqrt{\tau_0}.
    \end{align}

    We  divide $\gamma(s)$ into two paths  $\gamma_1$ and $\gamma_2$ such that
    $\gamma_1:[0, \tau']$ with $\gamma_1(0)=o$ and $\gamma_1(\tau')=\sigma_{o_{\tau'}}$,
    and  $\gamma_2:[\tau', \tau_0]$ with $\gamma_2(\tau')=o_{\tau'}$ and $\gamma_2(\tau_0)=p_{\tau}$.
   Then
    \begin{align*}
        L(p_\tau, \tau_0) &= \mathcal{L}(\gamma_1) + \mathcal{L}(\gamma_2),
    \end{align*}
    where
    \begin{align}\label{gamma-1}
        \mathcal{L}(\gamma_1) = \int_0^{\sqrt{\tau'}} \left( 2u^2 + \frac{1}{2}|\dot{\bar \gamma}|^2 + 2f(\bar \gamma(u))\right)du - 2\sqrt{\tau'}f(o_\tau'),
    \end{align}
    and
    \begin{align*}
        \mathcal{L}(\gamma_2) = \int_{\sqrt{\tau'}}^{\sqrt{\tau_0}} \left( 2u^2 + \frac{1}{2}|\dot{\bar \gamma}|^2 + 2f(\bar \gamma(u))\right)du +  2\sqrt{\tau'}f(o_\tau') - 2\sqrt{\tau_0}f(p_\tau).
    \end{align*}
     Since $o_\tau' $ lies in $\partial  K$,  there exists a uniform constant $C_2>0$, such that $f(o_\tau')\leq C_2 $. Thus by \eqref{L-ell-bound},  we   obtain
    \begin{align}\label{gamma2-bound}
        \mathcal{L}(\gamma_2) &\leq 2A_0\sqrt{\tau_0}-\mathcal{L}(\gamma_1)\notag\\
        & = 2A_0\sqrt{\tau_0}- \int_0^{\sqrt{\tau'}} \left( 2u^2 + \frac{1}{2}|\dot{\bar \gamma}|^2 + 2f(\bar \gamma(u))\right)du + 2\sqrt{\tau'}f(o_\tau')\notag\\
        &\leq 2(A_0+C_2)\sqrt{\tau_0}.
    \end{align}

 By (\ref{gamma-1}) and \eqref{tau'-tau0}, we also have
    \begin{align}\label{gamma1-lower-bound}
        \mathcal{L}(\gamma_1) &\geq \int_0^{\sqrt{\tau'}} 2u^2 - 2C_2\sqrt{\tau_0} \notag\\
        &\geq \frac{2}{3} \tau'^{\frac{3}{2}} - 2C_2\sqrt{\tau_0}\notag\\
        &\geq \frac{2}{3} \epsilon^{\frac{3}{2}} \tau_0^{\frac{3}{2}}- 2C_2\sqrt{\tau_0}.
    \end{align}
    Note that
    \begin{align}\label{gamma1-upper-bound}
        \mathcal{L}(\gamma_1) \leq L(p_\tau, \tau_0)  \leq 2A_0\sqrt{\tau_0}.
    \end{align}
   Clearly,  (\ref{gamma1-lower-bound}) is a contradiction with  \eqref{gamma1-upper-bound} when $\tau_0 \gg 1$. Thus (\ref{l-bound}) must be true.

Let  $\sigma(s)$ be  a minimal $\mathcal{L}$-geodesic from $(o_\tau', 0)$ to $(p_\tau, \tau_0-\tau')$. Note that $\gamma_2(\cdot)$ is also a curve from $o_\tau'$ to $p_\tau$.  Then
    \begin{align*}
        \lambda_{o_\tau'}(p_\tau, \tau_0-\tau') &= \frac{1}{2\sqrt{\tau_0-\tau'}}\int_0^{\tau_0-\tau'}\sqrt{s}\left(R_{g}+|\dot{\gamma}(s)-\nabla f|_{g}^2\right)(\sigma(s)) d s\notag\\
        & = \frac{1}{2\sqrt{\tau_0-\tau'}}\int_{\tau'}^{\tau_0}\sqrt{s-\tau'}\left(R_{g}+|\dot{\gamma}(s-\tau')-\nabla f|_{g}^2\right)(\sigma(s-\tau')) d s \notag\\
        &\leq \frac{1}{2\sqrt{\tau_0-\tau'}}\int_{\tau'}^{\tau_0}\sqrt{s-\tau'}\left(R_{g}+|\dot{\gamma}_2(s)-\nabla f|_{g}^2\right)(\gamma_2(s)) d s \notag\\
        & \leq \frac{1}{2\sqrt{\tau_0-\tau'}}\int_{\tau'}^{\tau_0}\sqrt{s}\left(R_{g}+|\dot{\gamma}_2(s)-\nabla f|_{g}^2\right)(\gamma_2(s)) d s \notag\\
        &=\frac{1}{2\sqrt{\tau_0-\tau'}}\mathcal{L}(\gamma_2).
        %&= \frac{2(A_0+C_2)\sqrt{\tau_0}}{2\sqrt{\tau_0-\tau'}}.
    \end{align*}
    Thus by  (\ref{gamma2-bound}) and  (\ref{l-bound}), it follows
       $$\lambda_{o_\tau'}(p_\tau, \tau_0-\tau') \leq 2(A_0+C_2)$$
    as long as $\tau_0\gg 1$.  Hence,   the proof of  lemma is complete.
\end{proof}

By Lemma \ref{ell-compact-set}, we know that  for any $\tau\gg 1$ there exists a $o_\tau'\in\partial  K$ and $\tau_0' = \tau_0-\tau'$, such that
\begin{align}\label{lambda-o-p}
    \lambda_{o_\tau'}(p_\tau, \tau_0')\leq C_0,
\end{align}
and
\begin{align}\label{tau0'}
    \frac{c}{2}\tau \leq \tau_0' \leq 2C\tau.
\end{align}

  By (\ref{lambda-o-p}) and (\ref{tau0'}),  we use the Perelman's argument \cite[Section 7]{P}  to derive the  upper bound  estimate of curvature   for the $\ell$-center $p_\tau$ in the following.   Namely, we prove

\begin{lem}\label{upper-boundofR} Let $p_r$  be the  $\ell$-center  determined in Proposition \ref{levelset-ell}.   Then there exists a uniform constant $C_0$ such that  for any $\tau\gg 1$ it holds
    \begin{align}\label{ell-center-curvature}
        R(p_\tau)\leq \frac{C_0}{\tau}.
    \end{align}
\end{lem}

\begin{proof}
Let $\Gamma(s): [0, \tau_0]$ be a minimal $\mathcal{L}$-geodesic with $ \Gamma(\tau_0) =\phi_{\tau_0}(p_\tau)=x_{\tau_0}$ and $ \Gamma(0) =o$.  Let $Y(s) = \frac{d\Gamma}{ds}$ and $\ell$ the corresponding reduced length.  Then by   \eqref{grad-ell} and \eqref{ell-harnack} in Lemma \ref{first-order-variation} (also see \cite[Section 25]{KL}),   for $x=\Gamma(\hat\tau)$  with  $\hat\tau\in [0,\tau_0]$,  we have
\begin{align}\label{gradient-ell-curvature}
    4\hat\tau |\nabla \ell|^2(x,\hat\tau) &= -4\tau R(x,-\hat\tau) + 4\ell(x,\hat\tau) -\frac{4}{\sqrt{\hat\tau}} \int_0^{\hat\tau} s^{\frac{3}{2}} H(Y(\Gamma(s))) ds\notag\\
    &+ \frac{4}{\sqrt{\hat\tau}}\int_0^{\hat\tau} \sqrt{s} R(\Gamma(s),-s) ds,
\end{align}
where
\begin{align}
    H(Y) = -\frac{\partial R}{\partial s}-2\langle Y, \nabla R\rangle+2 \operatorname{Ric}(Y, Y).\notag
\end{align}

By the isometry,
$(  \phi_{(-\hat\tau)^*}(Y), g = g(0))$ = $( Y,  g(-\hat\tau))$,  we see
\begin{align*}
    R_{\hat\tau}(x, -\hat\tau) = \frac{\partial R}{\partial\hat\tau}(x, -\hat\tau) = \langle \nabla R, \nabla f\rangle_{g}(\phi_{-\hat\tau}(x)),
\end{align*}
\begin{align*}
    \langle Y, \nabla R\rangle(x, \hat\tau) = \langle Y, \nabla R\rangle_{g(-\hat\tau)}(x) = \langle \phi_{(-\hat\tau)^*}(Y), \nabla R\rangle_{g}(\phi_{-\hat\tau}(x))
\end{align*}
and
\begin{align*}
    {\rm Ric}(Y, Y)(x, \hat\tau) = {\rm Ric}_{g(-\hat\tau)}(Y, Y)(x) = {\rm Ric}_{g}(\phi_{(-\hat\tau)^*}(Y),\phi_{(-\hat\tau)^*}(Y))(\phi_{-\hat\tau}(x)).
\end{align*}
Then
\begin{align*}
    H(Y)(x,\hat\tau) = &-\langle \nabla R, \nabla f\rangle_g(\phi_{-\hat\tau}(x))-2\langle \phi_{(-\hat\tau)^*}(Y), \nabla R\rangle_g(\phi_{-\hat\tau}(x))\\
    &+2 \operatorname{Ric}(\phi_{(-\hat\tau)*}(Y), \phi_{(-\hat\tau)*}(Y))(\phi_{-\hat\tau}(x))\\
    = &2 \operatorname{Ric}(\nabla f, \nabla f)(\phi_{-\hat\tau}(x))+4 \operatorname{Ric}(\phi_{(-\hat\tau)^*}(Y), \nabla f)(\phi_{-\hat\tau}(x))\\
    &+2 \operatorname{Ric}(\phi_{(-\hat\tau)^*}(Y), \phi_{(-\hat\tau)^*}(Y))(\phi_{-\hat\tau}(x))\\
    = &2 \operatorname{Ric}(\phi_{(-\hat\tau)^*}(Y)+\nabla f, \phi_{(-\hat\tau)*}(Y)+\nabla f)(\phi_{-\hat\tau}(x)).
\end{align*}
Moreover, by \eqref{lamda-function}, we have
\begin{align}\label{grad-ell-lambda}
    |\nabla \ell|^2_{g(-\hat\tau)}(x,\hat\tau) = |\nabla \lambda|^2_g(\phi_{-\hat\tau}(x), \hat\tau).
\end{align}
Thus by  \eqref{gradient-ell-curvature},  we get
\begin{align}\label{gradient-lambda-curvature}
    4\hat\tau |\nabla \lambda|^2(p,\hat\tau) &= -4\hat\tau R_g(p) + 4\lambda(p,\hat\tau)\notag \\
    &-\frac{8}{\sqrt{\hat\tau}} \int_0^{\hat\tau} s^{\frac{3}{2}} 2 \operatorname{Ric}_g(\phi_{(-s)^*}(Y)+\nabla f, \phi_{(-s)^*}(Y)+\nabla f)(\gamma(s)) ds\notag \\
    &+ \frac{4}{\sqrt{\hat\tau}}\int_0^{\hat\tau} \sqrt{s} R_g(\gamma(s)) ds,
\end{align}
where $p = \phi_{-\hat\tau}(x)$ and $\gamma(\hat\tau) = \phi_{-\hat\tau}(\Gamma(\hat\tau))$.

Recall that  the  minimal $\mathcal {L}$-geodesic $\gamma_2(s)=\phi_{-s}(\Gamma(s))$ ($s\in [\tau', \tau_0])$  is contained in  $M\setminus K$.  Then
\begin{align}\label{Harnack-positive}
    \operatorname{Ric}(\phi_{(-s)^*}(Y)+\nabla f, \phi_{(-s)^*}(Y)+\nabla f)(\gamma_2(s)) \geq 0.
\end{align}
Thus  for the  $\lambda$-function starting from $o_\tau'=\phi_{-\tau'}(\Gamma(\tau'))$, we get by \eqref{gradient-lambda-curvature},
\begin{align*}
    4\tau_0' |\nabla \lambda_{o_\tau'}|^2(p_\tau, \tau_0') \leq & -4\tau_0' R_g(p_\tau) + 4\lambda_{o_{\tau'}}(p_\tau,\tau_0')\notag \\
    & + \frac{4}{\sqrt{\tau_0'}}\int_0^{\tau_0'} \sqrt{s} R_g(\gamma_2(s)) ds\\
    \leq & -4\tau_0' R_g(p_\tau) + 4\lambda_{o_\tau'}(p_\tau) + 8\lambda_{o_\tau'}(p_\tau) \\
    = & -4\tau_0' R_g(p_\tau) + 12\lambda_{o_\tau'}(p_\tau),
\end{align*}
where $\tau_0'=\tau_0-\tau'$ and the second inequality comes from (\ref{L-length-stable}).
It follows
\begin{align}\label{Pel-ell-gradient}
    R(p_\tau)  \leq \frac{3\lambda_{o_\tau'}}{\tau_0'}.
\end{align}
Hence,  by \eqref{tau0'},  we  obtain (\ref{ell-center-curvature}).

\end{proof}

To get  lower bound estimate of scalar curvature, we need

\begin{lem}\label{curvature-lower-away}
    Let $(M,g)$ be a $\kappa$-noncollapsed steady gradient Ricci soliton with ${\rm Ric}\geq 0$ on $M\setminus K$. Suppose that   \eqref{linear-f} holds. Then there are  compact set $K'$ with $K\subset K'$ and  constant $c_0>0$ such that
    \begin{align}\label{upper-bound-R}
     1-R(x)\geq c_0>0, ~\forall~x\in M\setminus  K'.
     \end{align}
\end{lem}

\begin{proof}
    On the contrary,  we suppose that  \eqref{upper-bound-R} fails. Then there exists  a sequence of points $p_i\to \infty$ such that
    \begin{align*}
        R(p_i) \geq 1-\epsilon,
    \end{align*}
      where  $\epsilon>0$ is a small constant  to be determined lately.
   Thus  by \eqref{scalar-equ}, we  get
    \begin{align}\label{nabla-f-epsilon}
        |\nabla f|(p_i)\leq \sqrt{\epsilon}.
    \end{align}

    Denote $\sigma_i(t)$  to be the unit speed minimal geodesic with $\sigma_i(0) = o$ and $\sigma_i(s_i) = p_i$.
    We may assume that $p_i \in M\setminus K$. Let
    \begin{align*}
        s_i' = \inf\{t| ~\sigma_i(u) \in M\setminus K, ~\forall u>t\}
    \end{align*}
    and $p_i' = \sigma_i(s_i')$. Since $K$ is compact, there exists a constant $D>0$ such that $s_i'\leq D$ for all $i$ and $f(p_i')\leq D$ for all $i$.
    Thus by \eqref{nabla-f-epsilon}, we get
    \begin{align*}
        \langle \nabla f, \sigma'_i \rangle(p_i) \leq |\nabla f|(p_i) \leq \sqrt{\epsilon}.
    \end{align*}
    Hence, for any $t\in[s_i',s_i]$, we obtain
    \begin{align*}
        \langle \nabla f, \sigma'_i \rangle(\sigma_i(t)) &= \langle \nabla f, \sigma'_i \rangle(p_i) - \int_t^{s_i} \frac{d}{du}\langle \nabla f, \sigma'_i \rangle(\sigma_i(u)) du\\
        & = \langle \nabla f, \sigma'_i \rangle(p_i) - \int_t^{s_i} {\rm Ric}(\sigma'_i, \sigma'_i)(\sigma_i(u)) du\\
        & \leq \langle \nabla f, \sigma'_i \rangle(p_i)\\
        &\leq \sqrt{\epsilon}.
    \end{align*}
    Consequently,
    \begin{align}\label{f-less-growth}
        f(p_i) &= f(p_i') + \int_{s_i'}^{s_i} \frac{d}{du} f(\sigma_i(u)) du\notag\\
        &= f(p_i') + \int_{s_i'}^{s_i} \langle \nabla f, \sigma'_i \rangle(\sigma_i(u)) du\notag\\
        & \leq D + \sqrt{\epsilon}(s_i-s_i')\notag\\
        & \leq 10\sqrt{\epsilon}\rho(p_i),~\forall i\gg 1.
    \end{align}
    By choosing $\epsilon \leq \frac{c_1^2}{10000}$, \eqref{f-less-growth} contradicts to \eqref{linear-f}.  Therefore, we prove the lemma.

\end{proof}

By  Lemma  \ref{curvature-lower-away} and  Lemma \ref{uniform-shi}, we prove

  \begin{prop}\label{ell-curvature-lower}
      Let $(M,g)$ be a $\kappa$-collapsed steady gradient Ricci soliton with ${\rm Rm}\geq 0$ on $M\setminus K$. Suppose that   \eqref{linear-f} holds. Then there exists a constant $C>0$ such that
      \begin{align}\label{lower-scalar-R0}
          R(p)\geq \frac{C}{\rho(p)}
      \end{align}
      for all $\rho(p)>r_0$.
  \end{prop}

  \begin{proof} We use the argument  in the proof of \cite[Proposition 4.3]{DZ-TAMS}.   By   Lemma  \ref{curvature-lower-away},    we can choose    two compact sets  $\hat K$ and $ K'$  of $M$  with  $K\subset  K'\subset \hat K$ such that
       \begin{align}\label{upper-bound-R-1}
        1-R(x)\geq c_0>0, ~\forall~x\in \hat K\setminus  K',
        \end{align}
       where  $c_0$ is a  small constant.
       By a result of Chen \cite{Ch}, we may also assume
       \begin{align}\label{lower-bound-R}R(x)\ge c_0,  ~\forall~x\in \hat K\setminus  K'.
       \end{align}

        For  any $p\in M\setminus \hat K$, we have
      \begin{align*}
          \frac{d}{dt} R(\phi_t(p)) = {\rm Ric}(\nabla f, \nabla f)\geq 0,~\forall ~t\leq 0.
      \end{align*}
      It follows
      \begin{align*}
          0\leq R(\phi_t(p))\leq R(p),~\forall~ t\leq 0.
      \end{align*}
      Since
      \begin{align*}
          \frac{d}{dt} f(\phi_t(p)) = -|\nabla f|^2(\phi_t(p)),~\forall ~t\leq 0,
      \end{align*}
      by \eqref{scalar-equ}, we get
      \begin{align*}
          1-R(p)\leq -\frac{d}{dt} f(\phi_t(p))\leq 1,~\forall ~t\leq 0.
      \end{align*}
      Hence
      \begin{align}\label{f-R-time}
          (1-R(p))|t|\leq f(p) - f(\phi_t(p)) \leq |t|,~\forall~ t\leq 0.
      \end{align}

      By Lemma \ref{uniform-shi} and the evolution equation of scalar curvature, we have
      \begin{align*}
          \left|\frac{\partial}{\partial t} R^{-1}(p, t)\right| \leq \frac{|\Delta R(p, t)|}{R^2(p, t)}+\frac{2|\operatorname{Ric}(p, t)|^2}{R^2(p, t)} \leq C_0+2,
      \end{align*}
      and consequently,
      \begin{align}\label{R-t-lower}
          R(p, t)|t| \geq \frac{|t|}{(C_0+2)|t|+R(p, 0)^{-1}} \geq \frac{1}{2(C_0+2)}
      \end{align}
      for all $|t|\gg 1$.
  Since for any $p\in M\setminus \hat K$, there exists $x_p\in \hat K\setminus  K'$ and $t_p<0$ such that $\phi_{t_p}(x_p) = p$.   Thus by the first inequality in \eqref{R-t-lower} together with  the first inequality in  (\ref{f-R-time}),
  % together with \eqref{f-R-time},
   we get
     \begin{align}\label{R-estimate}
       R(p) & \geq \frac{1}{|t_p|} \cdot \frac{1}{(C_0+2)+\left(R\left(x_p\right)\left|t_p\right|\right)^{-1}} \notag \\
        & \geq \frac{1-R(x_p)}{f(p)-f\left(x_p\right)}   \cdot \frac{1}{(C_0+2)+\left(R\left(x_p\right)\left|t_p\right|\right)^{-1}}\notag \\
       & \geq \frac{1-R\left(x_p\right)}{2(f(p)-f(o))} \cdot \frac{1}{(C_0+2)+\left(R\left(x_p\right)\left|t_p\right|\right)^{-1}}.
          \end{align}
       Note that  $f(x_p)\leq C_0'$ for  $x_p\in \hat K\setminus  K'$.  Then  by  the second inequality in (\ref{f-R-time})
      together with  (\ref{lower-bound-R}), we have
        \begin{align*}
          |t_p|\geq f(p)-f(x_p)\geq  c_0\frac{f(p)-C_0'}{c_0} \geq \frac{1}{R\left(x_p\right)}
      \end{align*}
      as long as $\rho(p)\gg1.$
      Hence,  by inserting the above relation into (\ref{R-estimate}) together with (\ref{upper-bound-R-1}), we obtain
      $$R(p)  \geq  \frac{c_0}{2 c_2 (C_0+3) \rho(p)},$$
      where $c_2$ is the constant in \eqref{linear-f}.

  \end{proof}

\section{Construction  of  shrinking Ricci solitons}

In this section, we construct  the shrinking Ricci soliton via the blow-down  around $p_\tau$ by the estimates in Section 2, 3.

  Let $\{p_i\}\to \infty$  be any sequence in $M$. Set
  $\tau_i = f(p_i)$.  Then  by Proposition \ref{levelset-ell}, there are   $q_i\in\Sigma_{\tau_i}$ and $\tau_i'\in[c\tau_i,C\tau_i]$ such that
\begin{align}\label{lambda-qi}
    \lambda(q_i,\tau_i')\leq A_0,
\end{align}
where $c,C$ and $A_0$ are uniform constants.
Equivalently, we have
 \begin{align}\label{ell-xi}
    \ell(x_i, \tau_i')\leq A_0,
\end{align}
 where $x_i = \phi_{\tau_i'}(q_i)$.

With the help of  Lemma \ref{upper-boundofR}   and Proposition \ref{ell-curvature-lower},  we   can apply Proposition \ref{Dimension-reduction} to prove

\begin{lem}\label{local-type1-convergence}
    Let $x_i$, $\tau_i'$ defined as above. Then the sequence of rescaled Ricci flows $(M, \tau_i'^{-1}g(-\tau_i'+\tau_i' t); x_i)$ converges to $(N'\times \mathbb{R},g_\infty' = h'(t)+ds^2;x_\infty)$, $t\in(-\infty, 0]$, where $(N', h'(t))$ is a non-flat ancient  $\kappa$-solution.
\end{lem}

\begin{proof}
    By the isometry
    \begin{align}\label{flow-isometry}
        (g(-\tau_i'), x_{i}) \overset{\phi_{-\tau_i'}}{\cong} (g = g(0), q_{i}),
    \end{align}
the  rescaled Ricci flow $(M, \tau_i'^{-1}g(-\tau_i'+\tau_i' t); x_i)$   is   isometric to  $(M, \tau_i'^{-1}g(\tau_i' t); q_i)$.
By Lemma \ref{upper-boundofR}   and Proposition  \ref{ell-curvature-lower}, we see that  there exists a constant $C>0$ such that
     \begin{align}\label{R-tau}
        C^{-1} \leq \tau_i' R(q_{i}) \leq C.
    \end{align}
Thus the limit of rescaled Ricci flows $(M, \tau_i'^{-1}g(\tau_i' t); q_i)$ is isometric to one of $(M, R(q_{i})g(R(q_{i})^{-1} t); q_{i} )$.
     Moreover, by Proposition \ref{Dimension-reduction},   the limit of $(M, \tau_i'^{-1}g(-\tau_i'+\tau_i' t); x_i)$ is a split flow  $(g_\infty' = h'(t)+ds^2;x_\infty)$ on
    $N'\times \mathbb{R}$.  Since $R_{g_\infty'}(x_\infty)\geq C^{-1}$ by
     \eqref{R-tau}, $(N', h'(t))$ is a non-flat ancient $\kappa$-solution.
\end{proof}

   Let $g_i(t) = \tau_i'^{-1}g(\tau_i' t)$ and $\ell_i(x, \tau) = \ell(x, \tau_i' \tau)$, where $\tau = -t$.  Then $\ell_i(x, \tau)$ is  the reduced distance from $(o,0)$ w.r.t. the backward flow  $\hat g_i(\tau)=g_i(-\tau)$.  Moreover, by the scale invariant property and \eqref{ell-xi}, we have
\begin{align}\label{scaled-ell}
    \ell_i(x_i, 1) = \ell_{o}(x_i, \tau_i')\leq  A_0.
\end{align}
On the other hand,
by the convergence in Lemma \ref{local-type1-convergence},   for any fixed radius $D>0$, we have
\begin{align}\label{scaled-curvature}
    R(p, t)\leq C(D),
\end{align}
where  $(p, t)\in B_{g_i(-1)}(x_i, D)\times [-10, -1]$ when $i\gg1$.   By (\ref{scaled-ell}) and (\ref{scaled-curvature}), we  do  the derivative  estimate  for    $\ell_i(x,\tau)$ in the following.

\begin{prop}\label{local-ell-estimate}
    Let $\ell_i(x, \tau)$ be a sequence  of  reduced distance functions   defined as above.  Then  for any fixed $D>0$, there exists uniform constant $\bar C(D)>0$ such that
    \begin{align}\label{sequence-ell-bound}
        0\leq \ell_i(x,\tau) \leq \bar C(D),~\forall (x , t)\in B_{g_i(-2)}(x_i,D)\times [-10, -2]
    \end{align}
    and
    \begin{align}\label{sequence-ell-gradient-bound}
        |\frac{\partial \ell_i}{\partial \tau}(x,\tau)|+|\nabla \ell_i(x,\tau)|\leq \bar C(D),~\forall (x, t)\in B_{g_i(-2)}(x_i,D)\times [-8, -2]
    \end{align}
    for all $i\gg1$, where $\tau=-t$.
\end{prop}

\begin{proof}
     The nonnegativity of $\ell_i$ follows from the definition (\ref{L-length-soliton-0-hat}) and the nonnegativity of scalar curvature for ancient solutions. In the following, we  always  denote $C,  \bar C, C_i$ to be uniform constants only depending on $D$.

    Let $\Gamma_i(\tau)$ be the minimal $\mathcal{L}$-geodesics between $(o,0)$ and $(x_i,1)$. Then by \eqref{scaled-ell}, we have
    \begin{align*}
        \mathcal{L}(\Gamma_i(\tau))\leq 2A_0.
    \end{align*}
    Since the Ricci curvature  of $g_i(t)$ is nonnegative,  by the Ricci  equation  (\ref{ricci-equ}) it holds
    \begin{align}\label{geodesic-ball-harnack}
        B_{g_i(t)}(x_i,D)\subseteq B_{g_i(-1)}(x_i,D), ~\forall~t\leq -1.
    \end{align}
    On the other hand, by \eqref{scaled-curvature} and the distance distortion estimate we also have
    \begin{align}\label{geodesic-ball-distortion}
        B_{g_i(-1)}(x_i,D)\subseteq B_{g_i(t)}(x_i,C_0(D)), ~\forall~t\in[-10, -2].
    \end{align}

    Now we  fix $\tau\in [2, 10]$ and let $\sigma_i(s)$, $s\in[1,\tau]$, be the minimal geodesic from $x_i$ to $x$ w.r.t. $g_i(-\tau)$, where $x\in B_{g_i(-\tau)}(x_i,D)\subseteq B_{g_i(-2)}(x_i,D)\subseteq B_{g_i(-1)}(x_i,D)$. Then  $\sigma_i(s) \subseteq  B_{g_i(-1)}(x_i,D)$. Thus by \eqref{scaled-curvature} and the distance distortion estimate, we obtain
    \begin{align*}
        |\sigma_i'(s)|^2_{g_i(-s)} \leq e^{2nC(D)|s|}|\sigma_i'(s)|^2_{g_i(-\tau)} \leq C(D)|\sigma_i'(s)|^2_{g_i(t)}
    \end{align*}
    for all $s\in[1,\tau]$. Consequently,
    \begin{align}\label{L-combine}
        L_i(x, \tau) &\leq \mathcal{L}(\Gamma_i)+\int_1^\tau \sqrt{s}(R_{g_i(-s)}+|\sigma'(s)|^2_{g_i(-s)}) ds\notag\\
        &\leq \mathcal{L}(\Gamma_i)+\int_1^\tau \sqrt{s}(R_{g_i(-s)}+C(D)|\sigma'(s)|^2_{g_i(-\tau)}) ds\notag\\
        &\leq 2A_0 + 2\sqrt{\tau} C(D) + \frac{2}{3}C(D)\frac{D^2(\tau^\frac{3}{2}-1)}{(\tau-1)^2}\notag\\
        &\leq 2A_0+ C'(D).
    \end{align}
    Hence,  \eqref{sequence-ell-bound} follows from \eqref{L-combine}  together with \eqref{geodesic-ball-distortion}.

    Next we prove (\ref{sequence-ell-gradient-bound}).  Let $\tau\in[2,8]$ and $\bar \Gamma_i(s)$ be a minimal $\mathcal{L}$-geodesic between $(o,0)$ and $(x,\tau)$, where $x\in B_{g_i(-\tau)}(x_i,D)$. Set
    \begin{align*}
        \bar\tau_i' = \inf\{t~| ~\bar \Gamma_i(s) \in B_{g_i(-\tau)}(x_i,D),~\forall s>t\}
    \end{align*}
    and $q_i = \bar \Gamma_i(\bar\tau_i')$.
    By a change of variable $u = \sqrt{s}$,  we  write $\hat\Gamma_i(u) = \bar\Gamma_i(u^2)$. Then
    \begin{align*}
        \mathcal{L}(\hat\Gamma_i(u)) = \int_0^{\sqrt{\tau}} 2u^2R_{g_i(-u^2)}+\frac{1}{2}|\hat\Gamma'_i(u)|^2_{g_i(-u^2)} du.
    \end{align*}
    It follows
    \begin{align}\label{geodesic-tau-i'}
        \int_{\sqrt{\bar\tau_i'}}^{\sqrt{\tau}} \frac{1}{2}|\hat\Gamma'_i(u)|^2_{g_i(-u^2)} du\leq \mathcal{L}(\hat\Gamma_i(u)) = 2\sqrt{\tau}\ell_i(x,\tau)\leq C_1(D),
    \end{align}
    where the last inequality follows from \eqref{sequence-ell-bound}.
    On the other hand, by  the distance distortion estimates, we have
    \begin{align*}
        D^2 &= d_{g_i(-\tau)}^2(x,q_i)\leq (\int_{\sqrt{\bar\tau_i'}}^{\sqrt{\tau}} |\hat\Gamma'_i(u)|_{g_i(-\tau)} du)^2\\
        &\leq C(D)(\int_{\sqrt{\bar\tau_i'}}^{\sqrt{\tau}} |\hat\Gamma'_i(u)|_{g_i(-u^2)} du)^2\\
        &\leq C(D)(\sqrt{\tau}-\sqrt{\bar\tau_i'})\int_{\sqrt{\bar\tau_i'}}^{\sqrt{\tau}} |\hat\Gamma'_i(u)|^2_{g_i(-u^2)} du.
        \end{align*}
    Thus by (\ref{geodesic-tau-i'}), we get
    \begin{align}\label{tau-i'-lower-bound}
        \sqrt{\tau}-\sqrt{\bar\tau_i'}\geq C_3(D)
    \end{align}
    for all $x\in B_{g_i(-\tau)}(x_i,D)$ and $i\gg1$.

   We notice  that $\hat\Gamma'_i(u)$ satisfies   the minimal $\mathcal{L}$-geodesic equation,
    \begin{align*}
        \nabla_{\hat\Gamma'_i(u)}\hat\Gamma'_i(u)-2u^2\nabla R+4u{\rm Ric}(\hat\Gamma'_i(u))=0,  \forall ~u\in[\sqrt{\bar\tau_i'}, \sqrt{\tau}].
    \end{align*}
    Then by  Shi's estimates and the fact that ${\rm Ric}\geq 0$ on $B_{g_i(-1)}(x_i,D)$ for all $i\gg 1$, we  have
    \begin{align*}
        \frac{d}{du}|\hat\Gamma'_i(u)|^2_{g_i(-u^2)} &= 4u^2\langle\nabla R, \hat\Gamma'_i\rangle - 4u {\rm Ric}(\hat\Gamma'_i, \hat\Gamma'_i)\\
      &\le  4u^2|\langle\nabla R, \hat\Gamma'_i\rangle |_{g_i(-u^2)}\\
       &\leq C_4(D)|\hat\Gamma'_i(u)|_{g_i(-u^2)}\\
        &\leq 4C_4(D)(1+|\hat\Gamma'_i(u)|^2_{g_i(-u^2)}).
    \end{align*}
    Integrating the above inequality and  by \eqref{geodesic-tau-i'},  we obtain
    \begin{align*}
       & |\hat\Gamma'_i(\sqrt{\tau})|^2_{g_i(-\tau)}-|\hat\Gamma'_i(u)|^2_{g_i(-u^2)}
       \\&=  \int_{u}^{\sqrt{\tau}} \frac{d}{dv}|\hat\Gamma'_i(v)|^2_{g_i(-v^2)} dv\\
        &\leq 4\sqrt{\tau}C_4(D) + 8C_4(D)\int_{\sqrt{\bar\tau_i'}}^{\sqrt{\tau}} \frac{1}{2}|\hat\Gamma'_i(u)|^2_{g_i(-u^2)} du\\
        &\leq C_5(D).
 \end{align*}
Consequently,
\begin{align}\label{grdient-gamma}
   |\hat\Gamma'_i(u)|^2_{g_i(-u^2)}\ge   |\hat\Gamma'_i(\sqrt{\tau})|^2_{g_i(-\tau)}-C_5(D),  ~\forall ~u \in[\sqrt{\bar\tau_i'},   \sqrt{\tau}].
   \end{align}

     By \eqref{tau-i'-lower-bound} and  (\ref{grdient-gamma}),
      we have
    \begin{align*}
        \mathcal{L}(\bar \Gamma_i) &\geq \int_{\sqrt{\tau_i'}}^{\sqrt{\tau}} (\frac{1}{2}|\hat\Gamma'_i(u)|^2_{g_i(-u^2)} + 2u^2R_{g_i(-u^2)} )du\\
        &\geq \int_{\sqrt{\tau_i'}}^{\sqrt{\tau}} \frac{1}{2}|\hat\Gamma'_i(u)|^2_{g_i(-u^2)} du\\
        &\geq C_6(D)(\sqrt{\tau}-\sqrt{\tau_i'})|\hat\Gamma'_i(\sqrt{\tau})|^2_{g_i(-\tau)}-C_6(D)C_5(D)\\
        &\geq C_7(D)|\hat\Gamma'_i(\sqrt{\tau})|^2_{g_i(-\tau)}-C_8(D).
    \end{align*}
  On other hand,  by \eqref{sequence-ell-bound},  we have
    \begin{align*}
        \mathcal{L}(\bar \Gamma_i)\leq 2\sqrt{\tau} \ell_i(x,\tau)\leq 10C(D).
    \end{align*}
    Thus combining above two inequalities, we obtain
    \begin{align*}
        |\hat\Gamma'_i(\sqrt{\tau})|^2_{g_i(-\tau)}\leq C_9(D).
    \end{align*}
    By (\ref{grad-ell}), it follows
        \begin{align*}
        |\nabla \ell_i(x, \tau)|^2 = |\bar \Gamma'_i(\tau)|^2 = \frac{|\hat\Gamma'_i(\sqrt{\tau})|^2}{4\tau}\leq \frac{C_9(D)}{8}.
    \end{align*}
    Moreover, by \eqref{ell-equ} and \eqref{sequence-ell-bound},  we deduce
    \begin{align*}
        2|\frac{\partial}{\partial \tau} \ell_i(x, \tau)| \leq |\nabla \ell_i(x, \tau)|^2+ R + \frac{\ell_i(x,\tau)}{\tau}\leq \frac{C_9(D)}{8} + C(D).
    \end{align*}
Hence,   the above two relations give (\ref{sequence-ell-gradient-bound}).
 The proposition is proved.

 \end{proof}

\begin{rem}\label{remark}
    Since the Ricci curvature of $(M, g)$ is just nonnegative  outside the compact set $K$, we
   %  We note that we only assume positive , thus
    cannot use the global Harnack inequality directly  in \cite[Section 7]{P} to  get the gradient estimates for $\ell_i(x,\tau)$-function at any space-time $(x, \tau)$. We shall restrict the corresponding  minimal $\mathcal{L}$-geodesic on $M\setminus K$ and do the time gap estimate  for the  restricted   minimal $\mathcal{L}$-geodesic, see Lemma \ref{ell-compact-set} and  (\ref{tau-i'-lower-bound}).

\end{rem}

By  Lemma \ref{local-type1-convergence} and Proposition \ref{local-ell-estimate}, we are able to construct a limit  non-flat shrinking Ricci soliton via  the rescaled Ricci flows of $(M, g(t))$ by  following  the  strategy in  \cite[Section 9]{P}.

\begin{prop}\label{Pel-shrinker}
    Let $x_i$, $\tau_i'$  be chosen as in Lemma \ref{local-type1-convergence}. Then the sequence of rescaled Ricci flows $(M, \tau_i'^{-1}g(\tau_i' t); x_i)$ converges to $(N'\times \mathbb{R},g_\infty' = h'(t)+ds^2;x_\infty)$, $t\in(-\infty, -1]$, where $(N', h'(t))$ is a  non-flat  shrinking Ricci soliton.
\end{prop}

\begin{proof}
    By Lemma \ref{local-type1-convergence}, we know that the sequence of rescaled Ricci flows $(M, \tau_i'^{-1}g(\tau_i' t); x_i)$ converges to $(M_\infty=N'\times \mathbb{R},h'(t)+ds^2;x_\infty)$, $(-\infty, -1]$, where $(N', h'(t))$ is a  non-flat ancient $\kappa$-solution. We  only need to show $(N', h'(t))$ is indeed a shrinking Ricci soliton.

     By Proposition \ref{local-ell-estimate},  the sequence of functions $\{\ell_i\}$  converges subsequently to a function $\ell_\infty$ on $M_\infty\times [2,8]$ in the $C^\alpha_{loc}$ sense for any $\alpha\in(0,1)$.  Since $\ell_\infty$ is also locally  Lipschitz, which  is an element in $W^{1,2}_{loc}(M_\infty\times [2,8])$,   we may  assume that  $\{\ell_i\}$  converges weakly in $W^{1,2}_{loc}$ to $\ell_\infty$. Moreover, according to \cite[Section 9]{MT}, by the monotonicity of reduced volume, one can show the inequalities \eqref{ell-geq} and \eqref{ell-leq} in Lemma \ref{ell-differential-equation} for $\ell_\infty$ become  equalities  simultaneously. Namely,   $\ell_\infty$ satisfies
    \begin{align*}
        2 \frac{\partial l_{\infty}}{\partial \tau}+\left|\nabla l_{\infty}\right|^2-R_{g'_{\infty}}+\frac{l_{\infty}}{\tau}=0
    \end{align*}
    and
    \begin{align}\label{ell-diff-equation}
        2 \Delta_{g'_{\infty}} l_{\infty}-\left|\nabla l_{\infty}\right|^2+R_{g'_{\infty}}+\frac{l_{\infty}-n}{\tau}=0
    \end{align}
    in the distributional sense.

     Let
    \begin{align*}
        u = (4\pi \tau)^{-\frac{n}{2}}e^{-\ell_\infty}>0.
    \end{align*}
   Then  $u$ satisfies the conjugate heat equation
    \begin{align}\label{backward-limit}
        \frac{\partial u}{\partial \tau} - \Delta_{g'_{\infty}} u + R_{g'_{\infty}} u = 0
    \end{align}
    in the distributional sense. Thus the standard regularity theory gives smoothness of $\ell_\infty$. On the other hand,  if we  let
    $$v = \left(\tau\left(2 \Delta_{g'_{\infty}} l_{\infty}-\left|\nabla l_{\infty}\right|^2+R_{g'_{\infty}}\right)+l_{\infty}-n\right) u ,$$
    then  by \eqref{ell-diff-equation}, we have $v = 0$.
    Following  the Perelman's computation \cite[Proposition 9.1]{P},  we get by (\ref{backward-limit}),
    \begin{align*}
        \left(\frac{\partial}{\partial \tau}-\Delta_{g'_{\infty}}+R_{g'_{\infty}}\right) v=-2 \tau\left|{\rm Ric}_{g'_{\infty}}+\nabla^2 l_{\infty}-\frac{1}{2 \tau} g'_{\infty}(\tau)\right|^2 u
    \end{align*}
    on $M_\infty\times [2,8]$. Thus by $u>0$, we obtain
    \begin{align*}
        {\rm Ric}_{g'_{\infty}}+\nabla^2 l_{\infty}-\frac{1}{2 \tau} g'_{\infty}(\tau)=0
    \end{align*}
    on $M_\infty\times [2,8]$. This implies that $(M_\infty, g_\infty'(t); x_\infty)$ ($t\in[-8, -2]$)  is a shrinking Ricci soliton.  By  the uniqueness of Ricci flow and the splitting  structure $M_\infty=N'\times \mathbb{R}$,  we prove  the theorem immediately.
\end{proof}

By (\ref{flow-isometry}), we actually prove

\begin{theo}\label{asymtotic-shrinker}
    Let $q_i$, $\tau_i' $ be the sequences chosen  as in \eqref{lambda-qi}.  Then the sequence of rescaled Ricci flows $(M,g_{q_i}(t);q_i) $ converges to $(N'\times \mathbb{R},h'(t)+ds^2; q_\infty)$, $t\in (-\infty, 0]$ in the Cheeger-Gromov sense, where $(N', h'(t))$ is a non-flat shrinking Ricci soliton.
\end{theo}

\section{Proofs of  main theorems}

In this section, we prove main results in Introduction 0  by following the strategy in  \cite{ZZ-4d}. % Let us   first recall the following definition  introduced by Perelman (cf. \cite{P2,Yi-flyingwings,ZZ-4d}, etc.).

%\begin{defi}\label{epsilonclose}
    %For any $\epsilon>0$, we say a pointed Ricci flow $\left(M_1, g_1(t); p_1\right), t \in$ $[-T, 0]$ ($T$ may be the infinity), is $\epsilon$-close to another pointed Ricci flow $\left(M_2, g_2(t); p_2\right), t \in[-T, 0]$, if there is a diffeomorphism onto its image $\bar{\phi}: B_{g_2(0)}\left(p_2, \epsilon^{-1}\right) \rightarrow M_1$, such that $\bar{\phi}\left(p_2\right)=p_1$ and $\left\|\bar{\phi}^* g_1(t)-g_2(t)\right\|_{C^{\left[\epsilon^{-1}\right]}}<\epsilon$ for all $t \in\left[-\min \left\{T, \epsilon^{-1}\right\}, 0\right]$, where the norms and derivatives are taken with respect to $g_2(0)$.
%\end{defi}

% By the compactness result  for  rescaled  Ricci flows  in  Proposition \ref{Dimension-reduction},  we know that for any $\epsilon>0$, there exists a compact set  $D(\epsilon)>0$, such that for any $p\in M\setminus D$, $(M,g_p(t);p)$ is $\epsilon$-close to a splitting  flow $(h_p(t)+ds^2; p)$,  where  $h_p(t)$  is an $(n-1)$-dimensional ancient $\kappa$-solution.
%Since  the  $\epsilon$-close  splitting  flow $(h_p(t)+ds^2; p)$ may not be unique for a point  $p$,
 %we  may introduce a function on $M$ for each  $\epsilon$  as in \cite{Yi-flyingwings, ZZ-4d},
%\begin{align}\label{notation-f}F_{\epsilon}(p)=\inf_{h_p} \{ {\rm{Diam}} (h_p(0))\in(0,\infty)  \}.
%\end{align}
     %For simplicity, we always  omit the subscript $\epsilon$ in the function  $F_{\epsilon}(p)$ below.

\subsection{Proofs  of Theorem \ref{main-notype2} and Corollary \ref{alter-coro}}

As in \cite[Theorem 0.2]{ZZ-4d}, we use  the argument by contradiction to prove  Theorem \ref{main-notype2}.  Suppose that  there exists a sequence of normally scaled Ricci flows $(M, g_{p_i}(t); p_i)$ $(p_i\to \infty)$, which converges to a limit Ricci flow $(N\times \mathbb{R}, h(t) + ds^2; p_\infty)$, where $(N, h(t), p_\infty)$ is a compact ancient $\kappa$-solution of type II.  Then by \cite[Lemma 2.2]{ZZ-4d},    under the condition  of ${\rm Rm} \geq 0$ and ${\rm Ric}>0$ on $M\setminus K$ and the existence of    compact  split (not Ricci-flat)   ancient $\kappa$-solution $h(t)$,  the scalar curvature of $(M, g)$  decays to zero uniformly, i.e.
  \begin{align}\label{R-decay}
       \lim_{x \rightarrow \infty}R(x) = 0.
   \end{align}
Thus by the normalization  identity (\ref{scalar-equ}), it follows
\begin{align}\label{gradient-esti}
|\nabla f(x)|\to 1~{\rm as}~ \rho(x)\to\infty.
\end{align}
  Moreover, by \cite[Lemma 2.2]{DZ-SCM} (or \cite[Theorem 2.1]{CDM}), $f$ satisfies (\ref{linear-f}).  Hence,  all results in  Section  1-4   hold.

\begin{proof}[Proof of Theorem \ref{main-notype2}]
    Choose a sequence of  $t_k\to -\infty$.  Then by \cite[Lemma 4.3]{ZZ-4d},  for the  compact ancient $\kappa$-solution  $h(t)$ of type II, we know
        \begin{align}\label{diameter-to-infty}
            \lim_{k\rightarrow\infty}{\rm Diam}(h(t_k))R_h^{\frac{1}{2}}(p_\infty,t_k)\geq\lim_{k\rightarrow\infty}{\rm Diam}(h(t_k))R_{h,min}^{\frac{1}{2}}(t_k)=\infty,
        \end{align}
      where   $R_{h,min}(t)=\min \{R_h(\cdot, t)\}$  and $R_h(\cdot, t)$ denote the scalar curvature of $(N, h(x,t))$.
        Thus   there is  $t_0\in\{t_k\}$ for some $k_0$ such that
            \begin{align}\label{type2-increasing}
                {\rm Diam}(h(t_0))R_h^{\frac{1}{2}}(p_\infty,t_0)>100C_0A,
            \end{align}
where  the  large constants  $C_0$   and $A$  will be determined  lately.

      Recall that
      $(M,g_{p_i}(t);p_i)$ is $\epsilon$-close to $(N\times \mathbb{R},h(t)+ds^2; p_{\infty})$ when $i\gg1$.
      Set $T=t_{0}R^{-1}(p_i)$ and choose  $\epsilon<-\frac{1}{100t_{0}}$.
      We consider  the  split ancient  flow   $(\tilde N, \tilde h(t))$  of codimension one   for  the  normally rescaled flow
       $( M, g_{(\phi_T(p_i))}(t); $
      \newline $\phi_T(p_i))$  as in  Proposition \ref{Dimension-reduction}.
      Notice  that
        \begin{align}\label{metric-scale} &R(\phi_T(p_i))\bar g(T)= R(p_i,T)\bar g(T)\notag\\
           &=R_{g_{p_i}}(p_i,t_{0})\bar g_{p_i}(t_{0}),
           \end{align}
           where $\bar g(T)$ is the induced metric of $g$ on the level set  $f^{-1}(f(\phi_T(p_i)))$ and $\bar g_{p_i}(t)$ is the induced metric of $g_{p_i}(t)$ on the level set  $f^{-1}(f(p_i))$, respectively.
           Then according to  the proof of \cite[Proposition 3.6]{ZZ-4d},
   %     ${\rm Diam}(\tilde N, \tilde h(0))$ is $\epsilon$-close   to ${\rm Diam}( N,$  $~R_h(p_\infty, t_0) h(t_0))$.
       $(\tilde N,   \tilde h(0))$ is in fact
      a  limit of hypersurfaces
         $$( f^{-1}(f(\phi_T(p_i))), R(\phi_T(p_i))\bar g(T); \phi_T(p_i)).$$
            Thus by the convergences  of  $g_{p_i}(t)$,
            $\tilde h(0)=R_h(p_\infty,t_0)h(t_0)$.
    Consequently,  by  (\ref{type2-increasing}),
     we get
          \begin{align}\label{qi-sequence-large}
             & {\rm Diam}(R(\phi_T(p_i))\bar g(T))= {\rm Diam}(R(p_i,T)\bar g(T))\notag\\
           &= {\rm Diam}(\bar g_{p_i}(t_{0}))R_{g_{p_i}}^\frac{1}{2}(p_i,t_{0})
%= {\rm Diam}(R(\phi_T(p_i))\bar g(T))  \notag\\
>90C_0A
          \end{align}
          as long as  $i\gg1$.
         %  It follows
          %\begin{align}\label{level-diam}
             %&={\rm Diam}(\bar g(T))R^\frac{1}{2}(\phi_T(p_i))\notag\\
              %&={\rm Diam}(\bar g(T))R^\frac{1}{2}(p_i,T)\notag\\
              %&>90C_0A,~ i\gg1.
         % \end{align}

        %  Let %  In fact, it is  a limit of hypersurfaces
         %$$( f^{-1}(f(\phi_T(p_i))), R(\phi_T(p_i))\bar g(T); \phi_T(p_i)).$$
         % In particular,  $(\tilde N, \tilde h(0))$ is compact.
          % Thus by (\ref{type2-increasing}),
          % (\ref{level-diam}),
           %we get
         %$${\rm Diam}(\tilde N, \tilde h(0))\ge 90C_0A.$$
         %By the Definition \ref{notation-f}, it follows
         %\begin{align}\label{max-diameter-T1}
          %   F(\phi_{T}(p_i))>50C_0A, ~ i\gg1.
       %  \end{align}

         By Proposition \ref{levelset-ell}, on each level set $\Sigma_{f(\phi_{T}(p_i))}$, there exists  $q_i \in \Sigma_{f(\phi_{T}(p_i))}$ and  $\tau_i\in[cf(\phi_{T}(p_i)), Cf(\phi_{T}(p_i)) ] $ such that $\lambda(q_i, \tau_i) \leq A_0$, where $c, C$ and $A_0$ are uniform constants in Proposition \ref{levelset-ell}.  Clearly,  $\tau_i\to \infty$  since $f(\phi_{T}(p_i)) \to \infty$ when $i\to \infty$.  Thus  we can apply Theorem  \ref{asymtotic-shrinker} to see that  the sequence of normally rescaled flows $(M, g_{q_i}(t); q_i)$ converges to a nonflat shrinking gradient Ricci soliton $(N'\times \mathbb{R}, h'(t) + ds^2, q_\infty)$.  It follows that   $(N', h'(t))$ is also a non-flat shrinking soliton with ${\rm Rm} \geq 0$.
         Furthermore, by  \cite[Lemma 2.6 and Remark  2.8]{ZZ-4d},  $(N', h'(t))$ is compact as same as $(\tilde N, \tilde h(t))$.   Hence,  by the classification result of   compact $\kappa$-noncollapsed shrinking Ricci solitons   with ${\rm R_m}\ge 0$ (cf. \cite[Theorem 7.34]{CLN}, \cite{BW}, \cite{Ni}), there exists a  large constant  $A$ such that
               \begin{align}\label{diameter-N}
             {\rm Diam}(N', h'(0)) \leq A,
         \end{align}
         and
         \begin{align*}
             R_{h'(0)} \equiv 1.
         \end{align*}

  Note that  as same as  $(\tilde N,   \tilde h(0))$,   $(N',  h'(0))$ is the limit of   hypersurfaces
         $$( f^{-1}(f(\phi_T(p_i))), R(q_i))\bar g(T); q_i).$$
        Thus  by the above convergence
         %of   $(M, g_{q_i}(t); q_i)$
         and (\ref{diameter-N}),   we see
         \begin{align}\label{type1-sequence}
             {\rm Diam}(R(q_i)\bar g(T)) = {\rm Diam}(\bar g(T))R^{\frac{1}{2}}(q_i) \leq 2A.
         \end{align}
         On the other hand,  by  \cite[Lemma  1.3 and Lemma  2.6]{ZZ-4d},   there exists a  large constant  $C_0$ such that
 \begin{align}\label{type1-curvature-control}
             \frac{1}{C_0}R(x)\leq R(q_i) \leq C_0R(x),~\forall x\in \Sigma_{f(\phi_{T}(p_i))}
         \end{align}
         for all $i\gg 1$.
       Hence,   combining  \eqref{type1-sequence} and \eqref{type1-curvature-control}, we get
         \begin{align}\label{type1-2-diam}
             {\rm Diam}(R(\phi_T(p_i))\bar g(T))&={\rm Diam}(\bar g(T))R^\frac{1}{2}(\phi_T(p_i))\notag\\
             &\leq C_0{\rm Diam}(\bar g(T)) R^\frac{1}{2}(q_i)\notag\\
             &\leq 2C_0A,
         \end{align}
         which contradicts to \eqref{qi-sequence-large}.   Therefore,  the theorem is proved.
\end{proof}

\begin{proof}[Proof of Corollary \ref{alter-coro}]
    By Theorem \ref{main-notype2},  $(n-1)$-dimensional  compact split limit ancient flow $(N, h(t))$ of type II can't arise  from the blow-down of $(M,g)$. Thus we  need only  to show  that the compact split ancient flows of type I and noncompact split ancient flows   can't  occur  simultaneously from the blow-down of $(M,g)$.  In fact,   the latter  is true  by \cite[Theorem 1.2]{ZZ-high}.   The corollary  is  proved.
\end{proof}

\subsection{Proof of Theorem \ref{compact-linear-decay}}

As  in Theorem \ref{main-notype2}, we see that  (\ref{R-decay}),
     (\ref{gradient-esti}) and (\ref{linear-f}) are  all satisfied. Thus all results in  Section  1-4  hold.

\begin{proof}[Proof of Theorem \ref{compact-linear-decay}]
    The first inequality in \eqref{R-decay-linearly} has been proved in Proposition   \ref{ell-curvature-lower}.  We  only need to prove the   second inequality.   Suppose that it is not true,  then  there exists a sequence of points $p_i'\to \infty$  such that
    \begin{align}\label{curvature-arbitrary-large}
        R(p_i')\geq \frac{i}{\rho(p_i')}.
    \end{align}

    Let $s_i = f(p_i')$. Then  by (\ref{linear-f}),  there exists a uniform constant $C_1>0$ such that
    \begin{align}\label{si-levelset}
        C_1^{-1}\rho (p_i') \leq s_i \leq C_1\rho (p_i').
    \end{align}
    By Proposition  \ref{levelset-ell} and
      Lemma \ref{upper-boundofR},  for each $i\gg 1$ there exists $q_i\in \Sigma_{s_i}$  such that
    \begin{align}\label{ell-center-curvature-bound}
        R(q_i)\leq \frac{C_2}{s_i}
    \end{align}
    for some uniform constant $C_2>0$. On the other hand, by Corollary \ref{alter-coro}, we know that  $(n-1)$-dimensional  split limit solution $(N', h'(t))$  of rescaled Ricci flows $(M, g_{p_i'}(t); p_i')$ is a compact ancient $\kappa$-solution of type I. Thus  as in  (\ref{diameter-N}), we  have
    \begin{align*}
        {\rm Diam}(N',h'(0))\leq A.
    \end{align*}
      It follows that  there exists uniform constant $C_0>0$ as in  (\ref{type1-curvature-control}) such that
    \begin{align}\label{pi'-qi}
        C_0^{-1} R(q_i)\leq R(p_i')\leq C_0 R(q_i), ~\forall ~i\gg 1.
    \end{align}
 Hence, combining \eqref{si-levelset}, \eqref{ell-center-curvature-bound} and \eqref{pi'-qi},  we obtain
    \begin{align}\label{pi'-curvature-bound}
        R(p_i')\leq \frac{C_3}{\rho(p_i')}
    \end{align}
    for some uniform constant $C_3>0$.  But  this is a contradiction with  \eqref{curvature-arbitrary-large}   when $i\gg 1$.  The theorem is proved.

\end{proof}

 % \newpage

  \end{document}